\documentclass[10pt]{amsart}

\usepackage{soul} 
\usepackage[dvipsnames]{xcolor}
\usepackage[bookmarks,colorlinks=true,citecolor=OliveGreen,linkcolor=RoyalBlue]{hyperref}
\usepackage{amssymb,amsthm}
\usepackage[inline]{enumitem}
\usepackage{mathrsfs}
\usepackage{mathtools}
\usepackage{xparse}
\usepackage{mathabx}
\usepackage{nicefrac}
\usepackage{thmtools}
\usepackage[capitalise]{cleveref}		
\usepackage{bm}
\usepackage{lmodern}

\usepackage{tikz}
\usetikzlibrary{angles,quotes}

\usepackage[capitalise]{cleveref}		

\makeatletter
\def\thmt@refnamewithcomma #1#2#3,#4,#5\@nil{%
  \@xa\def\csname\thmt@envname #1utorefname\endcsname{#3}%
  \ifcsname #2refname\endcsname
    \csname #2refname\expandafter\endcsname\expandafter{\thmt@envname}{#3}{#4}%
  \fi
}
\makeatother

\declaretheorem[numberwithin=section]{theorem}
\declaretheorem[sibling=theorem]{proposition}

\declaretheorem[sibling=theorem]{lemma}

\declaretheorem[sibling=theorem,style=definition]{definition}

\declaretheorem[sibling=theorem,style=remark]{remark}

\declaretheorem[sibling=theorem,style=remark]{fact}
\declaretheorem[sibling=theorem,style=remark]{notation}

\newcounter{claimCounter}[theorem]

\declaretheorem[sibling=claimCounter,style=remark]{claim}

\newcounter{subClaimCounter}[claimCounter]

\newcommand{\cat}{\widehat{\phantom{\alpha}}}
\newcommand{\uhr}{\!\upharpoonright \!}

\newcommand\seq[1]{\langle #1 \rangle}

\renewcommand\=[1]{\textup{\textrm{#1}}}

\renewcommand{\phi}{\varphi}

\title[Hyperarithmetic exceptional directions]{Hyperarithmetic directions can all be exceptional for Marstrand's projection theorem}

\author{Noam Greenberg}
\author{Daniel Turetsky}

\address{School of Mathematics and Statistics \\
Victoria University \\
PO Box 600, Wellington, 6140\\
New Zealand}

\begin{document}

\begin{abstract}
We show that there is a $\Pi^0_2$ subset~$B$ of the Euclidean plane, of Hausdorff dimension 1, such that for every line $\ell$ through the origin with hyperarithmetic direction, the projection of~$B$ onto~$\ell$ has Hausdorff dimension~0, thus is exceptional for Marstrand's projection theorem. It follows that for any computable ordinal~$\alpha$, being $\alpha$-random does not guarantee the ``almost-all'' statement of the projection theorem.
\end{abstract}

\maketitle


\section{Introduction}

A central result in geometric measure theory is Marstrand's Projection Theorem, relating the Hausdorff dimension of a set in the plane with the Hausdorff dimension of its projections onto various lines.

\begin{theorem}[Marstrand's Projection Theorem~\cite{Marstrand:projection}]\label{thm:marstrand}
Suppose $B \subseteq \mathbb{R}^2$ is an analytic set. Let $\alpha = \={dim}_H(B)$ be the Hausdorff dimension of $B$. Then:
\begin{enumerate}
\item For almost every line through the origin, the Hausdorff dimension of the projection of $B$ onto that line has Hausdorff dimension  $\min\{1, \alpha\}$; and
\item If $\alpha > 1$, then for almost every line through the origin, the projection of $B$ onto that line has positive measure in the line.
\end{enumerate}
\end{theorem}

Here we are identifying a line through the origin with a point on the unit circle~$S^1$, which we call a \emph{direction} (more properly, we should rather consider the projective line, but that detail will not matter for this paper), and from that identifying it with an angle, i.e., the amount of radians around the circle. Thus a set of lines corresponds to a set of real numbers, and so it makes sense to discuss ``almost every line''. 

One response to this theorem is to study its limits --- that is, ``playing at the edges''. This has taken a variety of forms. We will be focused on conclusion (1) from Marstrand's theorem, though many of the questions asked could be similarly asked for conclusion (2).  There have been various generalizations of Marstrand's theorem to $\mathbb{R}^n$ and linear subspaces \cite{Mattila:Grassmannians}, and notions of fractal dimension other than Hausdorff dimension (see for example \cite{Jarvanpaa:packing,FalconerHowroyd:packing} for packing dimension), but we will restrict ourselves to the version presented here (for a general survey see \cite{60years,70years}). 

The first thing one might wonder is about the requirement that $B$ be analytic. Using the Continuum Hypothesis, Davies~\cite{Davies:projection} built a set such that conclusion (1) fails. It follows that some assumption on $B$ is required. Using the stronger set theoretic assumption of $V=L$, Slaman and Stull (unpublished), and independently Richter~\cite{Richter:coanalytic}, showed that a Davies-style counterexample can be made co-analytic. It remains open whether any counterexamples of this sort can be built within ZFC.

In the other direction, N.\ Lutz and Stull~\cite{LutzStull} used algorithmic information theory to give a new proof of Marstrand's theorem, and with this showed that the assumption that $B$ is analytic can be replaced with the assumption that the Hausdorff and packing dimensions of $B$ are equal.

Another approach is to consider the set of exceptional directions, i.e., the directions $\vec{e}\in S^1$ such that the projection onto the line in direction $\vec{e}$ (passing through the origin), $\={Proj}_{\vec{e}}(B)$, has small Hausdorff dimension. Kaufman's proof of the projection theorem~\cite{Kaufman} implies that for $\beta < \={dim}_H(B)$,
\[
	\={dim}_H \{ \vec{e} : \={dim}_H(\={Proj}_{\vec{e}}(B)) \le \beta\} \le \beta.
\]
This has been improved in various ways, with an optimal result recently obtained by Ren and Wang~\cite{RenWang}, confirming a conjecture of Oberlin~\cite{Oberlin} (for a generalization to higher dimensions see~\cite{CholakEtAl}). Other results consider packing dimension and Baire category~\cite{Orponen}.

We are interested in an approach inspired by computable analysis and algorithmic randomness. There, when considering a classical result which has a conclusion that holds almost everywhere, it is common to ask what level of algorithmic randomness is necessary to ensure the conclusion. Examples of such analysis include the Lebesgue Density Theorem~\cite{DenjoyDemuthDensity,Khan:density,DayMiller:cupping,OWrandomness,Cupping:BSL,MiyabeNiesZhang}, Birkhoff's Ergodic Theorem~\cite{Vyugin:ergodic,LaurentEtAl:Birkhoff,FranklinEtAl:Birkhoff}, and the statement that a Lipschitz function is almost everywhere differentiable~\cite{BrattkaEtAl:diff}. Thus, we are interested in the question ``how random does a direction need to be to ensure that conclusion (1) holds for that direction?''

The aforementioned new proof by N.\ Lutz and Stull shows the following:

\begin{theorem}
	If $B\subseteq \mathbb{R}^2$ is $\Sigma^0_2$, then for any 2-random direction $\vec{e}$, 
	\[
		\={dim}_H(\={Proj}_{\vec{e}}(B)) = \min \left\{ \={dim}_H(B), 1   \right\}. 
	\]
\end{theorem}

One might conjecture the obvious extrapolation, that $\alpha$-randomness suffices when $B$ is $\Sigma^0_\alpha$. Our main result (\cref{thm:main}) implies a strong failure of this. 

\begin{theorem}
	There is a $\Pi^0_2$ set $B\subseteq \mathbb{R}^2$ such that $\={dim}_H(B)=1$, but for any computable ordinal $\alpha < \omega_1^{ck}$, there is an $\alpha$-random direction $\vec{e}$ with $\={dim}_H(\={Proj}_{\vec{e}}(B))=0$. 
\end{theorem}

\begin{proof}
	By \cref{thm:main}, there is a  $\Pi^0_2$ set $B\subseteq \mathbb{R}^2$ such that $\={dim}_H(B)=1$, but for any hyperarithmetic $\vec{e}$, $\={dim}_H(\={Proj}_{\vec{e}}(B))=0$. For any computable $\alpha<\omega_1^{ck}$, there is a hyperarithmetic $\vec{e}$ that is $\alpha$-random. 
\end{proof}

\section{Preliminaries}

We refer the reader to \cite{Falconer:FractalGeometry,Mattila:textbook} for background on classical Hausdorff dimension.

\begin{definition}
A {\em direction} in $\mathbb{R}^2$ is a unit vector $\vec{e}$.

{\em Projecting} a set $B \subseteq \mathbb{R}^2$ onto direction $\vec{e}$ means forming the set
\[
	\={Proj}_{\vec{e}}(B) = \{ p \cdot \vec{e} : p \in B\}.
\]
\end{definition}

\begin{notation}
For a real number $\gamma$, let $\vec{u}_\gamma= (\cos \gamma\pi, \sin\gamma\pi)$.
\end{notation}

Using the above, conclusion (1) of \cref{thm:marstrand} can be rephrased as:
\begin{center}
For almost every $\gamma \in \mathbb{R}$, $\={dim}_H(\={Proj}_{\vec{u}_\gamma}(B)) = \min\{1, \={dim}_H(B)\}$.
\end{center}

We will work simultaneously in $2^\omega$ and the unit interval $[0,1]$, relying on the map $\iota: 2^\omega \to [0,1]$ to move between them, where $\iota$ is given by
\[
	\iota(x) = \sum_n x(n)\cdot 2^{-(n+1)}.
\]
Of course, $\iota$ is not exactly a homeomorphism, as it fails to be injective. The points where this failure occurs are those $x \in 2^\omega$ which are finite or cofinite, corresponding to those $\iota(x)$ which are dyadic rationals. We will restrict ourselves to the complement of these points, and so $\iota$ will be a computable injection with a uniformly computable inverse on the points we are working with.

Note that if $|\sigma| = n$, then $\iota([\sigma])$ is an interval of length $2^{-n}$.

We will also use $\iota$ to denote the induced map $2^\omega \times 2^\omega \to [0, 1] \times [0, 1]$, relying on context to indicate which map is intended.

\subsection{Jump Structures}

We refer the reader to \cite{Sacks:book,ChongYu:book} for background on the hyperarithmetic degrees.

\begin{definition}
Let $L$ be a computable linear order with a distinguished least element 0, a computable successor relation, and such that the set of limit points (nonzero points with no predecessor) is computable. For such a linear order~$L$, for any $a\in L$ which is not the greatest element of~$L$ (if it exists), we write $a+1$ for the successor of~$a$ in~$L$. 

A {\em jump structure} on $L$ is a family of sets $\overline{J} = (J_a)_{a \in L}$ satisfying:
\begin{itemize}
\item $J_0 = \emptyset$;
\item For all $a+1$, $J_{a+1} = J_a'$;
\item For each limit point $b$, $J_b = \bigoplus_{a <_L b} J_a$.
\end{itemize}
\end{definition}

\begin{fact}
Recall that a \emph{Harrison order} is a computable linear order which is ill-founded, but which has no hyperarithmetic infinite descending sequences. Such orders exist, and indeed, there is a Harrison order~$H$ on which a jump structure exists \cite{Harrison}. If $\overline{J}$ is any jump structure on a Harrison order, then for any $a$ in the ill-founded part of the order, $J_a$ computes all hyperarithmetic sets. Finally, the collection of jump structures on a given Harrison order is $\Pi^0_2$.
\end{fact}

\subsection{Effective Hausdorff Dimension}
We refer the reader to \cite{Li.Vitanyi:93,DH:Book,Nies:Book} for background on algorithmic randomness and Kolmogorov complexity.

For each $m$ (only $m = 1$ and $m=2$ will be relevant for this paper), fix some computable listing $(q_i)_{i \in \omega}$ of $\mathbb{Q}^m$.

\begin{definition}
For any $D \in 2^\omega$, $p \in \mathbb{R}^m$ and $\delta > 0$, the {\em $D$-computable Kolmogorov complexity of $p$ with precision $\delta$} is
\[
	K_\delta^D(p) = \min\{ K^D(i) : |p - q_i| < \delta\}.
\]
The {\em $D$-computable dimension of $p$} is
\[
	\={cdim}^D(p) = \liminf_{\delta \to 0} \frac{K_\delta^D(p)}{-\log_2 \delta}.
\]
\end{definition}

Note that if $2^{-n} \le \delta < 2^{-(n+1)}$, then $K^D_{2^{-n}}(p) \le K^D_{\delta}(p)$ and $-\log_2\delta \le n+1$. Thus
\[
	\liminf_{\delta \to 0} \frac{K^D_\delta(p)}{-\log_2 \delta} \ge \liminf_{n\to\infty} \frac{K^D_{2^{-n}}(p)}{n+1} = \liminf_{n\to\infty}\frac{K^D_{2^{-n}}(p)}n = \liminf_{n\to\infty}\frac{K^D_{2^{-n}}(p)}{-\log_2 2^{-n}},
\]
so it suffices to only consider $\delta$ of the form $2^{-n}$.

Note also that for any $p \in \mathbb{R}^m$, $K^D_{2^{-n}}(p) \le^+ mn + 2\log n$. 

For points $p \in 2^\omega\times 2^\omega$, we could make a similar definition using a computable dense enumeration of points in $2^\omega\times 2^\omega$, but the definition can be considerably simplified because $p$ is uniquely expressed in binary.
\begin{definition}
For any $D \in 2^\omega$ and $p \in 2^\omega \times 2^\omega$, let $p = (x, y)$. The {\em $D$-computable dimension of $p$} is
\[
	\={cdim}^D(p) = \liminf_{n \to \infty} \frac{K^D( x\uhr n, y\uhr n)}{n}.
\]
\end{definition}
An analogous definition holds for higher powers of Cantor space, but we have stated the version we will require.

Computable dimension and Hausdorff dimension are connected through the point-to-set principle of J.\ Lutz and N.\ Lutz \cite{LutzLutz:point-to-set}, which we will state for the special cases of $2^\omega \times 2^\omega$ and $\mathbb{R}^2$

\begin{theorem}[Point-to-Set Principle]
For any $A \subseteq 2^\omega \times 2^\omega$ or $A \subseteq \mathbb{R}^2$:
\[
	\={dim}_H(A) = \min_{D \in 2^\omega} \sup_{p \in A} \={cdim}^D(p).
\]
\end{theorem}

We will refer to those oracles $D$ where this minimum is achieved as oracles {\em witnessing} the Hausdorff dimension of $A$.

\subsection{Two Applications of Point-to-Set}
To illustrate the technique, we prove two folklore results using the Point-to-Set Principle.

\begin{proposition}\label{prop:iota_preserves_dimension}
For any $A \subseteq 2^\omega\times 2^\omega$, $\={dim}_H(A) = \={dim}_H(\iota(A))$ and with the same witnessing oracles.
\end{proposition}

\begin{proof}
By the Point-to-Set Principle, it suffices to show that for any $D \in 2^\omega$ and $p \in 2^\omega \times 2^\omega$, $\={cdim}^D(p) = \={cdim}^D(\iota(p))$. So fix $p$ and $D$, and let $p = (x, y)$. 

Note that the map $j: 2^{<\omega}\times 2^{<\omega} \to \mathbb{Q}^2$ given by
\[
	j(\sigma, \tau) = \left( 0.\sigma,0.\tau  \right)
\]
(where $0.\sigma = \sum_{n < |\sigma|} \sigma(n)\cdot 2^{-(n+1)}$)
is computable. Further, $|\iota(p) - j(x\uhr n, y\uhr n)| \le \sqrt{2} \cdot 2^{-n}$. Thus
\[
	K^D_{\sqrt{2}\cdot 2^{-n}}(\iota(p)) \le^+ K^D(x\uhr n, y\uhr n),
\]
where the additive constant is for the instruction to locate the $i$ with $q_i = j(x\uhr n, y\uhr n)$. It follows that
\begin{align*}
\={cdim}^D(\iota(p)) &= \liminf_{\delta \to 0} \frac{K^D_\delta(\iota(p))}{-\log_2\delta}\\
 &\le \liminf_{n \to \infty} \frac{K^D_{\sqrt{2}\cdot 2^{-n}}(\iota(p))}{-\log_2 (\sqrt 2\cdot 2^{-n})}\\
&= \liminf_{n \to \infty} \frac{K^D_{\sqrt{2}\cdot 2^{-n}}(\iota(p))}{n - 1/2}\\
& \le \liminf_{n \to \infty} \frac{K^D(x\uhr n, y\uhr n)}{n - 1/2}\\
&= \liminf_{n \to \infty} \frac{K^D(x\uhr n, y\uhr n)}{n} = \={cdim}^D(p).
\end{align*}

Conversely, suppose $|\iota(p) - q_i| < 2^{-n}$. Then $q_i$ is contained in some dyadic square of sidelength $2^{-n}$, but because $q_i$ might be near or on the boundary, there can be as many as four pairs $(\sigma, \tau) \in 2^n \times 2^n$ with $p \in [\sigma]\times[\tau]$. Given $i$ and $n$, we can uniformly obtain those four pairs. Thus
\[
	K^D(x\uhr n, y\uhr n) \le^+ K^D_{2^{-n}}(\iota(p)) + K^D(n) + 2.
\]
Here the 2 bits are necessary to indicate which of the four pairs is $(x\uhr n, y\uhr n)$.

Since $K^D(n) \le^+ 2\log n$, it follows that
\begin{align*}
\={cdim}^D(p) &= \liminf_{n \to \infty} \frac{K^D(x\uhr n, y\uhr n)}n\\
&\le \liminf_{n\to\infty} \frac{K^D_{2^{-n}}(\iota(p)) + K(n) + 2}n\\
&= \liminf_{n\to\infty} \frac{K^D_{2^{-n}}(\iota(p))}n\\
&= \={cdim}^D(\iota(p)).\qedhere
\end{align*}
\end{proof}

\begin{proposition}\label{prop:projection_loses_atmost_1}
For any set $B \subseteq \mathbb{R}^2$ and any direction $\vec{e}$,
\[
	\={dim}_H(B) \le \={dim}_H(\={Proj}_{\vec{e}}(B)) + 1.
\]
\end{proposition}

\begin{proof}
Fix $B$ and $\vec{e}$, and fix $D$ an oracle witnessing the Hausdorff dimension of $\={Proj}_{\vec{e}}(B)$.  Let $\alpha = \={dim}_H(\={Proj}_{\vec{e}}(B))$. We will show that for every $p \in B$, $\={cdim}^{D\oplus \vec{e}}(p) \le \alpha+1$, and so the result will follow by the Point-to-Set Principle.  So fix $p \in B$.

Let $\vec{e}^\perp$ be one of the two directions perpendicular to $\vec{e}$, noting that $\vec{e}^\perp$ is computable from $\vec{e}$.

By assumption, $\={cdim}^D(p\cdot \vec{e}) \le \alpha$. Note that if $q_i$ is within $2^{-n}$ of $p\cdot \vec{e}$, and $q_j$ is within $2^{-n}$ of $p\cdot \vec{e}^\perp$, then $q_i \vec{e} + q_j \vec{e}^\perp$ is within $\sqrt{2}\cdot 2^{-n}$ of $p$. Further, $q_i\vec{e} + q_j\vec{e}^\perp$ is uniformly computable from $i$ and $j$, with oracle $\vec{e}$. It follows that
\begin{align*}
	\={cdim}^{D\oplus \vec{e}}(p) &\le \liminf_{n\to\infty} \frac{K^{D\oplus \vec{e}}_{\sqrt{2}\cdot 2^{-n}}(p)}{-\log(\sqrt{2}\cdot 2^{-n})}\\
	&\le \liminf_{n\to\infty} \frac{K^D_{2^{-(n+1)}}(p\cdot \vec{e}) + K_{2^{-(n+1)}}(p\cdot \vec{e}^\perp)}{n-1/2}\\
	&\le \liminf_{n\to\infty} \frac{K^D_{2^{-(n+1)}}(p\cdot \vec{e}) + (n+1)  + 2\log(n+1)}{n+1}\\
	&= \liminf_{n\to\infty} \frac{K^D_{2^{-(n+1)}}(p\cdot \vec{e})}{n+1} + 1\\
	&= \={cdim}^D(p\cdot \vec{e}) + 1 \le \alpha+1.\qedhere
\end{align*}
\end{proof}

\section{A perverse set}

\begin{theorem}\label{thm:main}
There is a $\Pi^0_2$ class $B \subseteq [0, 1]\times [0, 1]$ with Hausdorff dimension 1, such that for every hyperarithmetic direction $\vec{e}$, $\={Proj}_{\vec{e}}(B)$ has Hausdorff dimension 0.
\end{theorem}

\begin{remark}
	In the language of \cite{FiedlerStull}, \cref{thm:main} says that the collection of hyperarithmetic directions is not unniversal for the class of $\Pi^0_2$ sets. 
\end{remark}

The majority of the proof is contained in the following lemma.

\begin{lemma}\label{lem:technical_pi01}
Given a $Z \in 2^\omega$ and $\bar{e} = (\vec{e}_s)_{s \in \omega}$ a sequence of directions in $\mathbb{R}^2$, there is $A \subset 2^\omega \times 2^\omega$ such that:
\begin{enumerate}
\item $A$ is uniformly $\Pi^0_1(Z, \bar{e})$;
\item For all $(x, y) \in A$, neither $x$ nor $y$ is finite or cofinite (thus $\iota$ is nice on $A$);
\item For all $(x, y) \in A$, for all $n$, $Z(n) = x((n+1)!) \ominus y((n+1)!)$;
\item $A$ has Hausdorff dimension at least 1; and
\item For every $s \in \omega$, $\={Proj}_{\vec{e}_s}(\iota(A))$ has Hausdorff dimension 0, with witnessing oracle $\emptyset$.
\end{enumerate}
\end{lemma}

Here $\ominus$ denotes symmetric difference, i.e., exclusive-or. The uniformity in~(1) means that there is a $\Pi^0_1$ set $W\subseteq (2^\omega)^2\times 2^\omega \times (S^1)^\omega$ such that for any $(Z,(\vec{e}_s))\in  2^\omega \times (S^1)^\omega$, the set $A = A(Z,(\vec{e}_s))$ is simply the section $\left\{ p \,:\,  (p,Z,(\vec{e}_s))\in W \right\}$. 

\begin{proof}[Proof of \cref{thm:main}, assuming \cref{lem:technical_pi01}]
Fix $H$, a computable Harrison order for which a jump structure exists.

Note that for any $X \in 2^\omega$, the statement ``$\Phi_i(X)$ is total and a direction'' is uniformly $\Pi^0_2$ in $i$ and $X$ (here $(\Phi_i)$ is an effective list of all Turing functionals). Thus, given $\overline{J}$, a jump structure on $H$, we can uniformly enumerate all directions computable from any $J_a$: for each $i$ and $a$, check the appropriate bit of $J_{a+2}$ to determine whether $\Phi_i(J_a)$ is total and a direction; if so compute $\Phi_i(J_a)$ and add it to the sequence. We let $\bar{e}(\overline{J})$ denote the sequence of directions obtained in this way from $\overline{J}$.

Given $Z \in 2^\omega$ and $\bar{e}$ a sequence of directions, let $A(Z, \bar{e})$ denote the set given by \cref{lem:technical_pi01}. We are ready to define our set $B$.
\[
	B = \bigcup \{ \iota\circ A(\overline{J}, \bar{e}(\overline{J})) : \text{$\overline{J}$ a jump structure on $H$}\}.
\]
As written, this might not appear $\Pi^0_2$. We give the following reformulation:
\begin{align*}
	p \in B \iff &\text{neither coordinate of $p$ is dyadic, and}\\
	& \text{for $(x, y) = \iota^{-1}(p)$ and $\overline{J} = \{ n \mapsto x((n+1)!)\ominus y((n+1)!)\}$}:\\
	&\hskip 1pc \bigl(\text{$\overline{J}$ is a jump structure on $H$}\bigr) \wedge \bigl((x,y) \in A(\overline{J}, \bar{e}(\overline{J}))\bigr)
\end{align*}
This is a conjunction of $\Pi^0_2$ and $\Pi^0_1$ statements.

\begin{claim}
For every hyperarithmetic direction $\vec{e}$, $\={Proj}_{\vec{e}}(B)$ has Hausdorff dimension 0.
\end{claim}

\begin{proof}
Fix such a $\vec{e}$. As $\vec{e}$ is hyperarithmetic, $\vec{e}$ is computable from $J_a$ for every jump structure $\overline{J}$ and any $a$ in the non-standard part of $H$. Thus $\vec{e}$ occurs in $\bar{e}(\overline{J})$ for any $\overline{J}$ a jump structure on $H$.

Fix $p \in B$, and fix $\overline{J}$ with $p \in \iota\circ A(\overline{J}, \bar{e}(\overline{J}))$. By \cref{lem:technical_pi01}, $\={cdim}(p\cdot \vec{e}) = 0$ (note the absence of an oracle). Since $p$ is arbitrary, it follows by the point-to-set principle that $\={Proj}_{\vec{e}}(B)$ has Hausdorff dimension 0.
\end{proof}

\begin{claim}
$B$ has Hausdorff dimension 1.
\end{claim}

\begin{proof}
Fix $\overline{J}$ some jump structure on $H$. Then $\iota\circ A(\overline{J}, \bar{e}(\overline{J})) \subseteq B$ has Hausdorff dimension at least 1, by \cref{lem:technical_pi01} and \cref{prop:iota_preserves_dimension}. It follows that $B$ has Hausdorff dimension at least 1. However, the Hausdorff dimension of $B$ can be no more than 1 by \cref{prop:projection_loses_atmost_1} (consider any hyperarithmetic direction). So $B$ has Hausdorff dimension 1.
\end{proof}
This completes the proof.
\end{proof}

It remains to prove \cref{lem:technical_pi01}.

\begin{proof}[Proof of \cref{lem:technical_pi01}]
We construct sets $U_n, V_n \subset 2^{<\omega} \times 2^{<\omega}$, for $n \in \omega$, such that:
\begin{itemize}
\item For each $(\sigma, \tau) \in U_n$, $|\sigma| = |\tau| = (n+1)!$;
\item For each $(\sigma, \tau) \in V_n$, $|\sigma| = |\tau| = (n+1)!+1$;
\item Every element of $U_{n+1}$ extends (coordinatewise) some element of $V_n$; and
\item Every element of $V_{n}$ extends (coordinatewise) some element of $U_{n}$.
\end{itemize}
Then $A$ will be defined as
\begin{align*}
	A &= \{ (x, y) \in 2^\omega \times 2^\omega : \forall n\ [(x\uhr (n+1)!, y\uhr (n+1)!) \in U_n]\}\\
	&= \{ (x, y) \in 2^\omega \times 2^\omega : \forall n\ [(x\uhr ((n+1)!+1), y\uhr ((n+1)!+1)) \in V_n]\}.
\end{align*}

We define now the $U_n$ and $V_n$.  We begin with $U_0 = 2^1 \times 2^1$.

Given $U_n$, our definition of $V_n$ depends on the value $n$:
\begin{itemize}
\item[($n$ is 0 mod 4)] Fix $b\in \{0,1\}$ such that $0\ominus b = Z(n)$. Define $V_n = \{ (\sigma 0, \tau b) : (\sigma, \tau) \in U_n\}$.
\item[($n$ is 1 mod 4)] Fix $b\in \{0,1\}$ such that $1\ominus b = Z(n)$. Define $V_n = \{ (\sigma 1, \tau b) : (\sigma, \tau) \in U_n\}$.
\item[($n$ is 2 mod 4)] Fix $b\in \{0,1\}$ such that $b\ominus 0 = Z(n)$. Define $V_n = \{ (\sigma b, \tau 0) : (\sigma, \tau) \in U_n\}$.
\item[($n$ is 3 mod 4)] Fix $b\in \{0,1\}$ such that $b\ominus 1 = Z(n)$. Define $V_n = \{ (\sigma b, \tau 1) : (\sigma, \tau) \in U_n\}$.
\end{itemize}
The only reason for this more complicated coding is that it simultaneously handles item (2) of \cref{lem:technical_pi01}.

\medskip

Given $V_n$, let $n = \seq{s, k}$, where $\seq{-,-}$ is a standard pairing function. 
We are interested in dot products of the form $p\cdot \vec{e}$, where $p$ is in the unit square and $\vec{e}$ is on the unit circle. These are effectively compact, and so by effective uniform continuity, we can effectively (relative to $\bar{e}$) find a $q \in \mathbb{Q}$ such that $d(p\cdot \vec{u}_q, p\cdot \vec{e}_s) < 2^{-(n+2)!}$ for all $p$ in the unit square. Fix such a $q$.

We wish $U_{n+1}$ to be a set of pairs $(\sigma, \tau)$ of the appropriate length, each coordinatewise extending some element of $V_n$, such that:
\begin{itemize}
\item If there is a rational of the form $r = m2^{-((n+1)!+3)}$, for $m \in \mathbb{Z}$, with
\[
	\={Proj}_{\vec{u}_q}(\iota([\sigma]\times[\tau])) \subseteq B(r, 2^{-((n+2)!-2)}),
\]
then $(\sigma, \tau) \in U_{n+1}$; and
\item If $(\sigma, \tau) \in U_{n+1}$, then there is a rational of the form $r = m2^{-((n+1)!+3)}$, for $m \in \mathbb{Z}$, with
\[
	\={Proj}_{\vec{u}_q}(\iota([\sigma]\times[\tau])) \subseteq B(r, 2^{-((n+2)!-3)}).
\]
\end{itemize}
Here $B(r, \delta)$ denotes the open ball centered at $r$ of radius $\delta$.

\begin{claim}
There is a choice of $U_{n+1}$ satisfying the above and uniformly computable from $V_n$ and $q$.
\end{claim}

\begin{proof}
We consider only $(\sigma, \tau)$ of the appropriate length and extending some element of $V_n$.

First note that $\iota([\sigma]\times[\tau])$ is contained in the unit square, and so every element $p$ has norm at most $\sqrt{2}$. Since $\vec{u}_q$ is a unit vector, $p\cdot \vec{u}_q$ has norm at most $\sqrt{2}$. Thus we can put a bound on the norm of a possible $r$, e.g., $|r| \le 2$. This means that there are only finitely many $m$ to check for any given $(\sigma, \tau)$.

Note that by convexity, $\={Proj}_{\vec{u}_q}(\iota([\sigma]\times[\tau])) \subseteq B(r, \delta)$ if and only if $p\cdot \vec{u}_q \in B(r, \delta)$ for each $p$ a corner of the square $\iota([\sigma]\times[\tau])$. Note also that for $p = (x, y)$, $p\cdot \vec{u}_q = x\cos q\pi + y\sin q\pi$.

So fixing a corner $p$ and an $r$ as above, we compute $d = |x\cos q\pi + y\sin q\pi - r|$ until we see either:
\begin{enumerate}
\item $d < 2^{-((n+2)!-3)}$; or
\item $d > 2^{-((n+2)!-2)}$.
\end{enumerate}
Observe that we will always eventually see at least one of these hold.

We define $U_{n+1}$ to be the set of $(\sigma, \tau)$ of the appropriate length and extending an element of $V_n$ for which there is an $r$ such that for all four corners of $\iota([\sigma]\times[\tau])$, we see (1) hold before we see (2) hold. Recall that there are only finitely many $r$ to consider, and so $U_{n+1}$ is computable as desired.

If $\={Proj}_{\vec{u}_q}(\iota([\sigma]\times[\tau])) \subseteq B(r, 2^{-((n+2)!-2)})$, for an appropriate $r$, then for this $r$ and all four corners $p$, we will have $d < 2^{-((n+2)!-2)}$; thus we will never see $(2)$ hold for one of the corners, and so we put $(\sigma, \tau)$ into $U_{n+1}$ when we eventually see $(1)$ hold for all four corners.

If $\={Proj}_{\vec{u}_q}(\iota([\sigma]\times[\tau])) \not \subseteq B(r, 2^{-((n+2)!-3)})$ for any appropriate $r$, then for every $r$ there is a corner $p$ with $d \ge2^{-((n+2)!-3)}$; thus we will never see $(1)$ hold for this corner, so we will reject $(\sigma, \tau)$ when we eventually see $(2)$ hold.

Thus $U_{n+1}$ is as desired.
\end{proof}

This completes the construction, and observe that items (1), (2) and (3) of \cref{lem:technical_pi01} are immediate. It remains only to show the appropriate Hausdorff dimensions.

\begin{claim}
For every $s \in \omega$, $\={Proj}_{\vec{e}_s}(\iota(A))$ has Hausdorff dimension 0 with witnessing oracle $\emptyset$.
\end{claim}

\begin{proof}
Fix a point $p \in A$. We will show that $\={cdim}(p\cdot \vec{e}_s) = 0$, and the claim will follow by the point-to-set principle.

Consider any $k$ with $n = \seq{s, k}$. By our action at stage $n$, there is some rational $r$ of the specified form with $d(p\cdot \vec{u}_q, r) < 2^{-((n+2)!-3)}$, and as previously noted, $|r| \le 2$. By the triangle inequality and our choice of $q$, $d(p\cdot \vec{e}_s, r) < 2^{-((n+2)!-4)}$.

Observe that $r$ is determined by $m$ and $n$, and $|m| \le 2^{(n+1)!+4}$. So $K(r) \in O((n+1)!)$. Thus
\begin{align*}
	\={cdim}(p\cdot \vec{e}_s) &= \liminf_{\delta \to 0} \frac{K_\delta(p\cdot \vec{e}_s)}{-\log_2\delta}\\
	&\le^\times \lim_{n \to \infty} \frac{(n+1)!}{(n+2)!-4} = 0.\qedhere
\end{align*}
\end{proof}

\begin{claim}\label{clm:full_axis_choices}
For every $(\sigma, \tau) \in V_n$, at least one of the following holds:
\begin{itemize}
\item For every $\sigma'$ extending $\sigma$ with $|\sigma'| = (n+2)!$, there is a $\tau'$ extending $\tau$ with $(\sigma', \tau') \in U_{n+1}$.
\item For every $\tau'$ extending $\tau$ with $|\tau'| = (n+2)!$, there is a $\sigma'$ extending $\sigma$ with $(\sigma', \tau') \in U_{n+1}$.
\end{itemize}
\end{claim}

\begin{proof}
The cases will depend on the angle of $\vec{u}_q$, for the $\vec{u}_q$ chosen in the construction of $U_{n+1}$. Without loss of generality, the smallest angle between $\vec{u}_q$ and the $x$-axis is at least $\pi/4$ (i.e., $\vec{u}_q$ is ``steep''). The argument for the other case will proceed the same, interchanging the roles of $\sigma$ and $\tau$.

Fix $\sigma'$ extending $\sigma$ with $|\sigma'| = (n+2)!$.  Note that $\iota([\sigma'] \times [\tau])$ is a rectangle with width $2^{-(n+2)!}$ and height $2^{-((n+1)!+1)}$.

For $r$ of the form $r = m2^{-((n+1)!+3)}$, let $\ell_r$ be the line $\{ p \in \mathbb{R}^2 : p\cdot \vec{u}_q = r\}$.  Note that our assumption of steepness gives us that this line has slope between $-1$ and $1$, and the slope does not depend on $r$. Without loss of generality, we will assume that the slope is nonnegative. The argument for a negative slope proceeds the same, except starting on the right edge of $\iota([\sigma']\times [\tau])$ instead of the left edge as we are about to do.

Note that any vertical line will intersect lines of the form $\ell_r$ with frequency at least $\sqrt{2}\cdot 2^{-((n+1)!+3)}$. That is, the vertical distance between successive such lines is at most this value (see \cref{fig:vert_dist}). Since $\sqrt{2}\cdot 2^{-((n+1)!+3)}$ is less than half the height of the rectangle $\iota([\sigma']\times[\tau])$, there is a line of the form $\ell_r$ which intersects $\iota([\sigma']\times[\tau])$ in the bottom half of the left edge. Fix such an $r$ and let $y_0$ be the $y$-value of this intersection, and let $\tau'$ be lexicographically least of length $(n+2)!$ such that $\min\iota([\tau']) \ge y_0$.

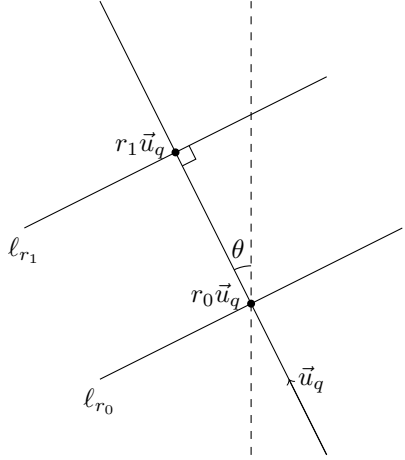
\begin{figure}
\begin{tikzpicture}
\draw[-] (0, 1) -- (-3, 7);
\draw[->] (0, 1) -- (-.5, 2) node[right] {$\vec{u}_q$};

\fill (-1,3) circle[radius=1.5pt] node[left, shift={(0,.1)}] {$r_0\vec{u}_q$};
\fill (-2, 5) circle[radius=1.5pt] node[left, shift={(0,.1)}] {$r_1\vec{u}_q$};

\draw[-] (1, 4) -- (-3, 2) node[below] {$\ell_{r_0}$};
\draw[-] (0, 6) -- (-4, 4) node[below] {$\ell_{r_1}$};

\node (D) at (-1.8, 5.1) {};
\node (E) at (-2,5) {};
\node (F) at (-1.9, 4.8) {};

\pic [draw, -, angle radius= 2mm] {right angle = D--E--F};

\draw[dashed] (-1,1) -- (-1,7);

\node (A) at (-1.5,4) {};
\node (B) at (-1,3) {};
\node (C) at (-1,4) {};

\pic [draw, -, "$\theta$", angle eccentricity=1.5] {angle = C--B--A};
\end{tikzpicture}
\caption{\label{fig:vert_dist} Here $r_0$ and $r_1$ are successive values of $r$, and so their difference is $2^{-((n+1)!+3)}$. Thus this is the distance between the points $r_0\vec{u}_q$ and $r_1\vec{u}_q$. By our assumptions, $-\pi/4 \le \theta \le 0$. Then the vertical distance between $\ell_{r_0}$ and $\ell_{r_1}$ is the hypotenuse of the triangle, which has length $2^{-((n+1)!+3)}\sec\theta \le \sqrt{2}\cdot 2^{-((n+1)!+3)}$.}
\end{figure}

We claim that this is our desired $\tau'$. Note that $\min \iota([\tau']) < y_0 + 2^{-(n+2)!}$, and so $\max \iota([\tau']) < y_0 + 2^{-(n+2)!+1} < y_0+2^{-((n+2)!-2)}$. Since $p \mapsto p\cdot \vec{u}_q$ is Lipschitz with constant 1, we have shown that for $p$ the bottom-left or top-left corner of $\iota([\sigma']\times[\tau'])$, $p\cdot \vec{u}_q \in B(r, 2^{-((n+2)!-2)})$.

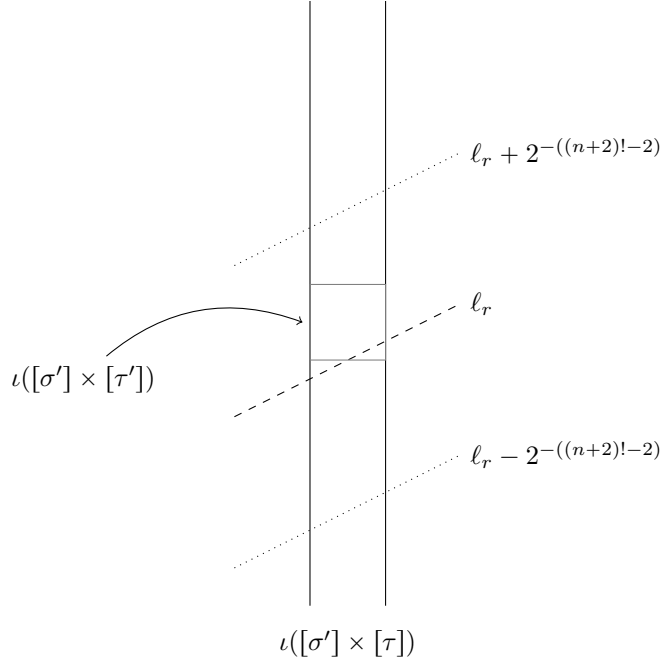
\begin{figure}
\begin{tikzpicture}

\draw[-] (0,0) -- (0, 8);
\draw[-] (1, 0) -- (1, 8);

\draw[dashed] (-1, 2.5) -- (2,4) node[right] {$\ell_r$};
\draw[dotted] (-1, 4.5) -- (2,6) node[right] {$\ell_r + 2^{-((n+2)!-2)}$};
\draw[dotted] (-1, .5) -- (2,2) node[right] {$\ell_r - 2^{-((n+2)!-2)}$};

\node at (.5, -.5) {$\iota([\sigma']\times[\tau])$};

\draw[-,gray] (0,3.25) -- (1,3.25) -- (1,4.25) -- (0,4.25) -- (0, 3.25);
\node (A) at (-3,3) {$\iota([\sigma']\times[\tau'])$};
\draw[->] (A) to [bend left] (-.1,3.75);

\end{tikzpicture}
\caption{\label{fig:square_fits} Here we see part of $\iota([\sigma']\times[\tau])$. Since the rectangle is so much taller than it is wide, neither the top or bottom edge appear in the figure. The square $\iota([\sigma']\times[\tau'])$ fits within $2^{-((n+2)!-2)}$ of $\ell_r$.}
\end{figure}

Let $y_1$ be the $y$-value of the intersection of $\ell_r$ with the right edge of $ \iota([\sigma']\times[\tau'])$. By our assumptions about the slope of $\ell_r$, we have $y_0 \le y_1 \le y_0 + 2^{-(n+2)!}$. Thus $y_1 - 2^{-(n+2)!} \le \min \iota([\tau']) \le \max \iota([\tau']) < y_1+2^{-((n+2)!-2)}$ (see \cref{fig:square_fits}). We have thus shown that for $p$ the bottom-right or top-right corner of $\iota([\sigma']\times[\tau'])$, $p\cdot \vec{u}_q \in B(r, 2^{-((n+2)!-2)})$. By convexity, $\iota([\sigma']\times[\tau']) \subseteq B(r, 2^{-((n+2)!-2)})$, and so $(\sigma', \tau') \in U_{n+1}$.
\end{proof}

\begin{claim}
For any oracle $D$, there is a point $p \in A$ with $\={cdim}^D(p) \ge 1$.
\end{claim}

\begin{proof}
Fix $D$. We construct $p$ in stages, making coherent choices of $(\sigma_n,\tau_n) \in U_n$ for each $n$. We begin with $(\sigma_0, \tau_0) \in U_0$ arbitrary.

Given $(\sigma_n, \tau_n) \in U_n$, let $(\hat{\sigma}, \hat{\tau})$ be the unique element of $V_n$ which extends $(\sigma_n, \tau_n)$ coordinatewise. Fix $\rho$ of length $(n+2)!-(n+1)!-1$ which is random without deficiency relative to $D$ and $(\hat{\sigma}, \hat{\tau}, K^D(\hat \sigma,\hat \tau))$. That is, for all $t \le (n+2)!-(n+1)!-1$, 
 \[
 K^{D}( (\rho\uhr t) |  (\hat{\sigma}, \hat{\tau}, K^D(\hat \sigma,\hat \tau)) \ge t,
 \]
where $K(a|b)$ is the conditional complexity of~$a$ given~$b$ (we could alternatively replace the oracle~$D$ by $D\oplus (\hat\sigma, \hat\tau, K^D(\hat\sigma,\hat\tau))$). Such $\rho$ exists by a measure counting argument. By Claim \ref{clm:full_axis_choices}, there is a $\nu$ with $(\hat{\sigma}\cat\rho, \hat{\tau}\cat\nu) \in U_{n+1}$ or $(\hat{\sigma}\cat\nu, \hat{\tau}\cat\rho) \in U_{n+1}$. We let $(\sigma_{n+1}, \tau_{n+1})$ be one of these. 

Fix $i$ with $(n+1)!+1\le i < (n+2)!$, and let $t= i - (n+1)! - 1$. Let $\nu_0$ and~$\nu_1$ be the strings (of length~$t$) such that $\sigma_{n+1}\uhr{i} = \hat\sigma\cat \nu_0$ and $\tau_{n+1}\uhr{i} = \hat\tau\cat \nu_1$. One of $\nu_0$ or $\nu_1$ is an initial segment of the $\rho$ that we chose. Now by symmetry of information (see for example \cite[Theorem\:3.9.1]{Li.Vitanyi:93}), 
\begin{align*}
	&\hskip -1 pc K^D(\sigma_{n+1}\uhr{i}, \tau_{n+1}\uhr{i})\\
	&=^+ K^D(\hat{\tau}, \hat{\sigma}) + K^D ((\nu_0, \nu_1) \,|\, \hat{\sigma},\hat{\tau}, K^D(\hat\sigma,\hat\tau)) 	\\
	&\ge^+ K^D(\hat{\tau},\hat{\sigma}) + K^D ((\rho\uhr{t}) \,|\, \hat{\sigma},\hat{\tau}, K^D(\hat\sigma,\hat\tau))
   \\
	&\ge K^D(\hat{\tau},\hat{\sigma}) + t.
\end{align*}

We let $x = \bigcup_n \sigma_n$ and $y = \bigcup_n \tau_n$ and $p = (x, y)$. Then $p\in A$; it remains to calculate its dimension.

By the above calculation, for any $i$, $K^D(x\uhr i, y\uhr i) \ge i - cn_i$, where $n_i$ is the largest $n$ with $(n+1)!+1 \le i$, and $c$ is a constant incorporating the constant for symmetry of information, the constant of the projection used in the first inequality above, and a constant accounting for the extra two bits when passing from $(\sigma_n, \tau_n)$ to $(\hat{\sigma}, \hat{\tau})$. Note that $n_i \in o(i)$.

Thus
\[
	\={cdim}^D(p) = \liminf_{i \to \infty} \frac{K^D(x\uhr i, y\uhr i)}{i} \ge \lim_{i \to \infty} \frac{i - cn_i}{i} = 1.\qedhere
\]
\end{proof}

It follows from the Point-to-Set Principle that $A$ has Hausdorff dimension at least 1.
\end{proof}

\bibliographystyle{plain}

\bibliography{unusual}

\end{document}